\documentclass[a4paper,11pt]{amsart}

\usepackage[margin=3cm]{geometry}
\usepackage{amsmath,amsthm,amssymb,mathtools,xcolor,enumitem,hyperref,%
            tikz}
\usepackage[scr]{rsfso}
\usepackage[british]{babel}

\usetikzlibrary{fadings}
\tikzfading[name=fade right,left color=transparent!0,right color=transparent!100]
\tikzfading[name=fade left,right color=transparent!0,left color=transparent!100]
\tikzfading[name=fade top,bottom color=transparent!0,top color=transparent!100]
\tikzfading[name=fade bottom,top color=transparent!0,bottom color=transparent!100]

\theoremstyle{definition}
\newtheorem{defn}{Definition}[section]
\newtheorem{ex}[defn]{Example}
\newtheorem*{ack}{Acknowledgements}
\theoremstyle{plain}
\newtheorem{lem}[defn]{Lemma}
\newtheorem{prop}[defn]{Proposition}
\newtheorem{thm}[defn]{Theorem}

\theoremstyle{remark}
\newtheorem{rmk}[defn]{Remark}

\counterwithin{equation}{section}
\setlist[enumerate]{label=\textup{(\roman*)}}

\newcommand{\Q}{\mathbb{Q}}
\newcommand{\R}{\mathbb{R}}
\newcommand{\Z}{\mathbb{Z}}
\newcommand{\Sb}{\mathbb{S}}

\DeclareMathOperator{\Comm}{Comm}

\DeclareMathOperator{\Sym}{Sym}
\DeclareMathOperator{\Cay}{Cay}

\makeatletter
\@namedef{subjclassname@2020}{%
  \textup{2020} Mathematics Subject Classification}
\makeatother

\title{Commensurators of abelian subgroups of biautomatic groups}
\author{Motiejus Valiunas}
\address{Instytut Matematyczny, Uniwersytet Wroc{\l}awski, pl.\ Grunwaldzki 2/4, 50-384 Wroc{\l}aw, Poland}
\email{valiunas@math.uni.wroc.pl}
\subjclass[2020]{20F10, 20E06}
\keywords{Biautomatic groups, commensurators, Leary--Minasyan groups}

\begin{document}

\begin{abstract}
We show that the commensurator of any finitely generated abelian subgroup $H$ in a biautomatic group centralises a finite-index subgroup of $H$. We deduce that the CAT(0) groups introduced by Leary--Minasyan \cite{leary-minasyan} are either biautomatic or cannot arise as subgroups of biautomatic groups, answering a question posed in \cite{leary-minasyan} and generalising an analogous result for Baumslag--Solitar groups. These are the first examples of CAT(0) groups that are not subgroups of biautomatic groups.
\end{abstract}
\maketitle

\section{Introduction}

The theory of automatic and biautomatic groups was developed in the 1980s, and is explored in a book by D.~B.~A.~Epstein et al.\ \cite{epstein}. Roughly speaking, a group $G$ is biautomatic if it can be equipped with a regular set of normal forms, in such a way that paths starting and ending at neighbouring vertices in a Cayley graph of $G$ fellow-travel; see Section~\ref{ssec:biauto} for a precise definition. Biautomaticity implies various geometric and algorithmic properties: for instance, a biautomatic group is finitely presented, satisfies a quadratic isoperimetric inequality, has solvable conjugacy problem, and finitely generated abelian subgroups of biautomatic groups are undistorted.

There has been substantial interest in biautomaticity of various classes of non-positively curved groups. In particular, word-hyperbolic groups \cite{cannon} and cubulated groups \cite{niblo-reeves}---that is, groups acting geometrically and cellularly on CAT(0) cube complexes---are biautomatic; more generally, Helly groups (introduced recently in \cite{ccgho}) are biautomatic. Artin groups of finite type \cite{charney92} and central extensions of word-hyperbolic groups \cite{neumann-reeves} are also biautomatic. For several decades, it has been an open question whether or not all CAT(0) groups---groups acting geometrically on CAT(0) spaces---are biautomatic. However, I.~J.~Leary and A.~Minasyan have recently \cite{leary-minasyan} constructed examples of CAT(0) groups that are not biautomatic.

More precisely, the paper \cite{leary-minasyan} studies commensurating HNN-extensions of $\Z^n$ (called \emph{Leary--Minasyan groups} in this paper), defined for a matrix $A \in GL_n(\Q)$ and a finite-index subgroup $L \leq \Z^n \cap A^{-1}(\Z^n)$ by the presentation
\[
G(A,L) = \langle x_1,\ldots,x_n,t \mid x_ix_j = x_jx_i \text{ for } 1 \leq i < j \leq n, t\mathbf{x}^{\mathbf{v}}t^{-1} = \mathbf{x}^{A\mathbf{v}} \text{ for } \mathbf{v} \in L \rangle,
\]
where we write $\mathbf{x}^{\mathbf{w}} := x_1^{w_1} \cdots x_n^{w_n}$ for any $\mathbf{w} = (w_1,\ldots,w_n) \in \Z^n$. It was shown \cite[Theorem~1.1]{leary-minasyan} that such a group $G(A,L)$ is CAT(0) if and only if $A$ is conjugate in $GL_n(\R)$ to an orthogonal matrix, and biautomatic if and only if $A$ has finite order. Thus, such groups provide first examples of CAT(0) groups that are not biautomatic.

A special case of Leary--Minasyan groups for $n = 1$ are the \emph{Baumslag--Solitar groups}, defined for $p,q \in \Z \setminus \{0\}$ by the presentation $BS(p,q) = \langle x,t \mid tx^pt^{-1} = x^q \rangle$. It is well-known that $BS(p,q)$ is biautomatic if $|p| = |q|$ (because it is cubulated, for instance), and that $BS(p,q)$ cannot be embedded in a biautomatic group if $|p| \neq |q|$ \cite[Corollary~6.8]{gersten-short}. This motivated a question \cite[Question~12.2]{leary-minasyan}, suggested by K.-U.\ Bux, which asks if a similar dichotomy is true for arbitrary Leary--Minasyan groups. We settle this question in the affirmative: that is, $G(A,L)$ is either biautomatic or not embeddable into a biautomatic group.

\begin{thm} \label{thm:cor}
Let $A \in GL_n(\Q)$, and let $L$ be a finite-index subgroup of $\Z^n \cap A^{-1}(\Z^n)$. Then $G(A,L)$ is a subgroup of a biautomatic group if and only if $A$ has finite order. In particular, there exist CAT(0) groups that are not embeddable into biautomatic groups.
\end{thm}

Theorem~\ref{thm:cor} is a consequence of Theorem~\ref{thm:main} below, which is a statement about commensurators of abelian subgroups in biautomatic groups.

Given a group $G$ and a subgroup $H \leq G$, we define the \emph{commensurator} $\Comm_G(H)$ of $H$ in $G$ as the set of elements $g \in G$ for which both $H \cap gHg^{-1}$ and $H \cap g^{-1}Hg$ have finite index in $H$; it is easily seen to be a subgroup of $G$. A related concept is that of \emph{abstract commensurator} $\Comm(H)$ of a group $H$, whose elements are equivalence classes of isomorphisms between finite-index subgroups of $H$, forming a group under composition (see Section~\ref{ssec:comm} for a precise definition). For any $H \leq G$, there is a canonical map $\Comm_G(H) \to \Comm(H)$ which sends $g \in \Comm_G(H)$ to the equivalence class of the isomorphism $\varphi\colon H \cap g^{-1}Hg \to H \cap gHg^{-1}$ defined as $\varphi(h) = ghg^{-1}$.

\begin{thm} \label{thm:main}
Let $G$ be a biautomatic group and let $H \leq G$ be a finitely generated abelian subgroup. Then the image of $\Comm_G(H)$ in $\Comm(H)$ is finite. In particular, there exists a finite-index subgroup $\Comm_G^0(H) \unlhd \Comm_G(H)$ such that every element of $\Comm_G^0(H)$ centralises some finite-index subgroup of $H$.
\end{thm}

Theorem~\ref{thm:main} can be seen as a generalisation of \cite[Theorem~1.2]{leary-minasyan}---the only difference between these two results is that $H$ is assumed to be finitely generated in the former and $\mathcal{M}$-quasiconvex (for some biautomatic structure $(Y,\mathcal{M})$ for $G$) in the latter; see Section~\ref{ssec:qconv} for a definition of $\mathcal{M}$-quasiconvex subsets. Indeed, $\mathcal{M}$-quasiconvexity implies finite generation for any subgroup of $G$, and so we may deduce \cite[Theorem~1.2]{leary-minasyan} from Theorem~\ref{thm:main}.

Note that Theorem~\ref{thm:cor} can be seen as an immediate corollary of Theorem~\ref{thm:main}. Indeed, if $G$ is biautomatic and $\widehat{G} \leq G$ is a subgroup isomorphic to a Leary--Minasyan group $G(A,L)$ (with an isomorphism sending $\widehat{t} \in \widehat{G}$ to $t \in G(A,L)$ and a subgroup $H < \widehat{G}$ to the subgroup $\langle x_1,\ldots,x_n \rangle < G(A,L)$), then we have $\widehat{t} \in \Comm_{\widehat{G}}(H) \leq \Comm_G(H)$. It then follows from Theorem~\ref{thm:main} that $t^k \in \Comm_G^0(H)$ for some $k \in \Z_{\geq 1}$, implying that $A^k = I_n$, and so that $A$ has order $\leq k$. Conversely, if $A$ has finite order then $G(A,L)$ is itself biautomatic by \cite[Theorem~1.1]{leary-minasyan}, and so it is a subgroup of a biautomatic group.

In fact, the class of CAT(0) groups not embeddable into biautomatic groups is more general than merely the groups $G(A,L)$. In particular, if $X$ is a proper CAT(0) space with no Euclidean factors, and if $K \leq \operatorname{Isom}(X)$ is a closed subgroup acting minimally and cocompactly on $X$, then one can show that a lattice $G$ in $\operatorname{Isom}(\mathbb{E}^n) \times K$ is not a subgroup of a biautomatic group unless $G$ has discrete image under the projection $\operatorname{Isom}(\mathbb{E}^n) \times K \to \operatorname{Isom}(\mathbb{E}^n)$. Indeed, it has been pointed out by S.~Hughes \cite[Theorem~7.7 and its proof]{hughes} that in this case there exists an element $t \in G$ that commensurates a subgroup $H \cong \Z^n$ of $G$, but $t^k$ does not centralise any finite index subgroup of $H$ (for any $k \geq 1$). The class of lattices in $\operatorname{Isom}(\mathbb{E}^n) \times K$, with $K$ as above, contains all CAT(0) Leary--Minasyan groups, and is studied in more detail by S.~Hughes in \cite{hughes}.

Our proof of Theorem~\ref{thm:main} relies on a triangulation of the sphere $\Sb^{N-1}$ associated to a biautomatic structure $(X,\mathcal{L})$ for $\Z^N$, described by W.~D.~Neumann and M.~Shapiro in \cite{neumann-shapiro}; see Remark~\ref{rmk:neumann-shapiro} for more details. We equip such a triangulation with an additional structure: namely, we define a \emph{polyhedral function} $f\colon \R^N \to \R$ (see Section~\ref{ssec:poly} for a definition) whose restriction to each polyhedral cone $\{ \beta \mathbf{v} \mid \mathbf{v} \in \Delta, \beta \in [0,\infty) \}$, where $\Delta \subseteq \Sb^{N-1}$ is a simplex in this triangulation, is linear and homogeneous. This function is chosen in such a way that it roughly approximates lengths of words in $\mathcal{L}$ representing an element of $\Z^N$. The key point of this construction is that, for a group $G$ with a biautomatic structure $(Y,\mathcal{M})$, it allows us to deal not only with an $\mathcal{M}$-quasiconvex abelian subgroup $\widehat{H} \leq G$, but also with any subgroup of such an $\widehat{H}$.

The structure of the paper is as follows. In Section~\ref{sec:prelim}, we introduce definitions and main results on commensurators, biautomatic groups and polyhedral functions. In Section~\ref{sec:geompoly}, we prove several results on polyhedral functions, and in Section~\ref{sec:biauto-poly} we associate a polyhedral function to a biautomatic structure for $\Z^N$ and we compare our construction to that of Neumann--Shapiro. In Section~\ref{sec:proof}, we use these results to prove Theorem~\ref{thm:main}.

\begin{ack}
I would like to thank Sam Hughes and the anonymous referee for their useful comments.
\end{ack}

\section{Preliminaries} \label{sec:prelim}

Throughout this paper, we denote by $\langle {-}, {-} \rangle$ the standard inner product on $\R^n$, and by $\| {-} \|$ the standard $\ell_2$-norm on $\R^n$, so that $\| \mathbf{v} \|^2 = \langle \mathbf{v},\mathbf{v} \rangle$ for all $\mathbf{v} \in \R^n$. We also write $I_n \in GL_n(\R)$ for the $n \times n$ identity matrix.

We denote by $[G : H]$ the index of a subgroup $H$ in a group $G$. We write $Z(G)$ for the centre of a group $G$, and $C_G(S)$ for the centraliser of a subset $S \subseteq G$.

\subsection{Commensurators} \label{ssec:comm}

Given a group $G$ and a subgroup $H \leq G$, the \emph{commensurator} of $H$ in $G$ is
\[
\Comm_G(H) := \{ g \in G \mid [H : H \cap gHg^{-1}] < \infty \text{ and } [H : H \cap g^{-1}Hg] < \infty \}.
\]
It is easy to check that $\Comm_G(H)$ is a subgroup of $G$ containing $H$.

A related notion is that of an abstract commensurator of a group $H$. In order to define it, let $\mathcal{C}_H$ be the set of all isomorphisms $\varphi\colon A \to B$, where $A$ and $B$ are finite-index subgroups of $H$. We say $\varphi,\psi \in \mathcal{C}_H$ are \emph{equivalent}, denoted $\varphi \sim_H \psi$ if there exists a finite-index subgroup $A$ of $H$ contained in the domains of both $\varphi$ and $\psi$ such that $\varphi(h) = \psi(h)$ for any $h \in A$; we denote by $[\varphi]$ or $[\varphi]_H$ the equivalence class of $\varphi$ in $\mathcal{C}_H$. We then define the \emph{abstract commensurator} of $H$ as
\[
\Comm(H) := \mathcal{C}_H / {\sim_H}.
\]
Given two isomorphisms $\varphi\colon A \to B$ and $\varphi'\colon A' \to B'$ between finite-index subgroups $A,A',B,B' \leq H$, we may define the product $[\varphi][\varphi']$ of the classes $[\varphi],[\varphi'] \in \Comm(H)$ to be the equivalence class of the map $\psi\colon (\varphi')^{-1}(A \cap B') \to \varphi(A \cap B')$ defined by $\psi(h) = \varphi(\varphi'(h))$. It is easy to verify that this makes $\Comm(H)$ into a group.

Now given an element $g \in \Comm_G(H)$ for groups $H \leq G$, we have $\varphi_g \in \mathcal{C}_H$, where $\varphi_g\colon H \cap g^{-1}Hg \to H \cap gHg^{-1}$ is the map defined by $\varphi_g(h) = ghg^{-1}$. Thus we have a canonical map $\Phi\colon \Comm_G(H) \to \Comm(H)$ sending $g \mapsto [\varphi_g]$, and this map can be easily checked to be a group homomorphism. It follows from the definitions that $g \in \ker(\Phi)$ precisely when $\varphi_g$ coincides with identity on some finite-index subgroup of $H \cap g^{-1}Hg$, which happens if and only if $g$ centralises a finite-index subgroup of $H$.

We will be interested in commensurators of finitely generated free abelian groups. In that case, it is easy to see that $\Comm(\Z^n)$ is isomorphic to $GL_n(\Q)$. Indeed, given a matrix $A \in GL_n(\Q)$, the intersection $L := \Z^n \cap A^{-1}(\Z^n)$ will have finite index in $\Z^n$, and we may define a map $\psi_A\colon L \to AL$ sending $\mathbf{v} \mapsto A\mathbf{v}$. This gives a map $GL_n(\Q) \to \Comm(\Z^n)$ sending $A \mapsto [\psi_A]$, which can be checked to be a group isomorphism.

We will also need to relate (abstract) commensurators of two groups one of which is a finite-index subgroup of another, and so we will use the following result. It is well-known, but we give a proof here for completeness.

\begin{lem} \label{lem:commfi}
Let $G$ be a group, and let $H,H' \leq G$ be subgroups such that $H' \leq H$ and $[H:H'] < \infty$. Then $\Comm_G(H) = \Comm_G(H')$, and there exists an isomorphism $\Psi\colon \Comm(H') \to \Comm(H)$ such that $\Phi = \Psi \circ \Phi'$, where $\Phi\colon \Comm_G(H) \to \Comm(H)$ and $\Phi'\colon \Comm_G(H') \to \Comm(H')$ are the canonical maps.
\end{lem}

\begin{proof}
Let $g \in G$, and denote by $K_\pm$ and $K_\pm'$ the groups $g^{\pm 1} H g^{\mp 1}$ and $g^{\pm 1} H' g^{\mp 1}$, respectively; note that $[K_\pm : K_\pm'] = [H : H'] < \infty$. If we have $[H' : H' \cap K_\pm'] < \infty$, then
\[
[H : H \cap K_\pm] \leq [H : H' \cap K_\pm'] = [H:H'] [H' : H' \cap K_\pm'] < \infty.
\]
Thus $\Comm_G(H') \subseteq \Comm_G(H)$. On the other hand, if $[H : H \cap K_\pm] < \infty$ then
\begin{align*}
[H' : H' \cap K_\pm'] &\leq [H : H' \cap K_\pm'] \\ &= [H : H \cap K_\pm] [H \cap K_\pm : H' \cap K_\pm] [H' \cap K_\pm : H' \cap K_\pm'] \\ &\leq [H : H \cap K_\pm] [H : H'] [K_\pm : K_\pm'] < \infty.
\end{align*}
Thus $\Comm_G(H) \subseteq \Comm_G(H')$, and so $\Comm_G(H) = \Comm_G(H')$.

Now any isomorphism $\varphi'\colon A' \to B'$ between finite-index subgroups of $H'$ is also an isomorphism between finite-index subgroups of $H$, and so represents a class $[\varphi']_H \in \Comm(H)$. Moreover, two such isomorphisms agree on a finite-index subgroup of $H'$ if and only if they agree on a finite-index subgroup of $H$. It follows that the map $\Psi\colon \Comm(H') \to \Comm(H)$, defined by $\Psi\left( [\varphi']_{H'} \right) = [\varphi']_H$, is well-defined and injective; furthermore, this map is easily seen to be a group homomorphism. Finally, given any isomorphism $\varphi\colon A \to B$ between finite-index subgroups of $H$ we may consider the map
\[
\varphi' = \varphi|_{A \cap H' \cap \varphi^{-1}(B \cap H')}\colon A \cap H' \cap \varphi^{-1}(B \cap H') \to B \cap H' \cap \varphi(A \cap H'),
\]
so that $[\varphi']_{H'} \in \Comm(H')$ and $[\varphi']_H = [\varphi]_H \in \Comm(H)$; thus $\Psi$ is surjective.

By construction, both $\Phi$ and $\Psi \circ \Phi'$ send an element $g \in \Comm_G(H) = \Comm_G(H')$ to the map represented by $\varphi_g\colon A \to gAg^{-1}, h \mapsto ghg^{-1}$ for some finite-index subgroup $A$ of $H$, implying that $\Phi = \Psi \circ \Phi'$, as required.
\end{proof}

\subsection{Biautomatic groups} \label{ssec:biauto}

Here we briefly introduce biautomatic groups and their main properties which we will be using in this paper. We refer the interested reader to \cite{epstein} for a more comprehensive account.

We fix a group $G$ with a finite generating set $Y$ of $G$; we view $Y$ as an abstract alphabet together with a fixed injective map $Y \hookrightarrow G$. We will always assume that $Y$ is \emph{symmetric} (i.e.\ the image $S$ of $Y$ in $G$ satisfies $S = S^{-1}$) and \emph{contains the identity} (i.e.\ contains an element $1 \in Y$ mapping to the identity $1_G \in G$). We denote by $Y^*$ the free monoid on $Y$, i.e.\ the set of all words with letters in $Y$, forming a monoid under concatenation. Given a word in $U \in Y^*$, we denote by $|U| \in \Z_{\geq 0}$ its length; furthermore, for any $t \in \Z_{\geq 0}$ we set $U(t)$ to be the prefix of $U$ of length $t$ if $t \leq |U|$, and we set $U(t) = U$ for $t > |U|$.

We have a monoid homomorphism $Y^* \to G$ induced by the inclusion $Y \hookrightarrow G$; we denote by $\overline{U} \in G$ the image of $U \in Y^*$ under this map, and we say that $U$ \emph{represents} $\overline{U}$. Given an element $g \in G$ we also denote by $|g|_Y$ the distance between vertices $1$ and $g$ in the Cayley graph $\Cay(G,Y)$; that is, $|g|_Y$ is equal to the length of the shortest word in $Y^*$ representing $g$, and so for any $U \in Y^*$ we have $\left| \overline{U} \right|_Y \leq |U|$.

A notion that appears in the theory of biautomatic groups is that of a \emph{finite state automaton} (\emph{FSA}). Roughly, a (deterministic) FSA $\mathfrak{A}$ over $Y$ is a finite directed multigraph $\Gamma$ with an assignment of a \emph{starting state} $v_0 \in V(\Gamma)$ and \emph{accept states} $\mathcal{A} \subseteq V(\Gamma)$, and with edges labelled by elements of $Y$ in such a way that for each vertex $v \in V(\Gamma)$ and each $y \in Y$, there exists at most one edge starting at $v$ and labelled by $y$; see \cite[Definition~1.2.1]{epstein} for a detailed definition. Given such an $\mathfrak{A}$, we say a subset $\mathcal{M} \subseteq Y^*$ is \emph{recognised} by $\mathfrak{A}$ if $\mathcal{M}$ is the set of words that label directed paths in $\Gamma$ starting at $v_0$ and ending in $\mathcal{A}$. A subset $\mathcal{M} \subseteq Y^*$ is said to be a \emph{regular language} if it is recognised by some FSA over $Y$.

This allows us to define biautomatic groups, as follows.

\begin{defn} \label{defn:biauto}
Let $G$ be a group, let $Y$ be a finite symmetric generating set of $G$ containing the identity, and let $\mathcal{M} \subseteq Y^*$. We say that $(Y,\mathcal{M})$ is a \emph{biautomatic structure} for $G$ if
\begin{enumerate}
\item $\mathcal{M}$ is a regular language; and
\item \label{it:biauto-ft} $\mathcal{M}$ satisfies the (\emph{two-sided}) \emph{fellow traveller property}: that is, there exists a constant $\lambda \geq 0$ such that if $U,V \in \mathcal{M}$ and $x,y \in Y$ are such that $\overline{U} = \overline{xVy}$, then $\left| \overline{U(t)}^{-1} \overline{x} \overline{V(t)} \right|_Y \leq \lambda$ for all $t \in \Z_{\geq 0}$.
\end{enumerate}
We say a biautomatic structure $(Y,\mathcal{M})$ for $G$ is \emph{finite-to-one} if for each $g \in G$ there exist only finitely many $U \in \mathcal{M}$ such that $\overline{U} = g$. We say $G$ is \emph{biautomatic} if it has a biautomatic structure.
\end{defn}

Condition~\ref{it:biauto-ft} in Definition~\ref{defn:biauto} has a more geometric description. In particular, if $U,V \in \mathcal{M}$ are such that $\overline{U} = \overline{xVy}$ for some $x,y \in Y$, then there exist paths $P_U$ and $P_V$ in $\Cay(G,Y)$ starting and ending distance $\leq 1$ away and labelled by $U$ and $V$, respectively. For any $t \in \Z_{\geq 0}$, we set $P_U(t)$ to be the initial subpath of $P_U$ of length $t$ if $t \leq |U|$, and we set $P_U(t) = P_U$ if $t > |U|$; we define $P_V(t)$ similarly. Condition~\ref{it:biauto-ft} then says that the endpoints of $P_U(t)$ and $P_V(t)$ are distance $\leq \lambda$ apart for any $t$.

It is known that every biautomatic group admits a finite-to-one biautomatic structure: see \cite[Theorem~2.5.1]{epstein}. For such a biautomatic structure, we have the following observation that will be crucial in our arguments.

\begin{lem}[see {\cite[Lemma~2.3.9]{epstein}}] \label{lem:simlengths}
Let $(Y,\mathcal{M})$ be a finite-to-one biautomatic structure on a group $G$. Then there exists a constant $\kappa \geq 0$ such that if $U,V \in \mathcal{M}$ and $x,y \in Y$ are such that $\overline{U} = \overline{xVy}$, then $|U|-\kappa \leq |V| \leq |U|+\kappa$. \qed
\end{lem}

\subsection{Quasiconvex subsets and subgroups} \label{ssec:qconv}

An important notion in biautomatic groups is that of quasiconvexity, defined as follows. We refer the interested reader to the paper \cite{gersten-short} by S.~M.~Gersten and H.~B.~Short for more details.

\begin{defn} \label{defn:qconv}
Let $(Y,\mathcal{M})$ be a biautomatic structure for a group $G$. We say a subset $S \subseteq G$ is \emph{$\mathcal{M}$-quasiconvex} if there exists a constant $\nu \geq 0$ such that any path in $\Cay(G,Y)$ that starts and ends at vertices in $S$ and is labelled by a word in $\mathcal{M}$ belongs to the $\nu$-neighbourhood of $S$.
\end{defn}

An important consequence of quasiconvexity is that if an $\mathcal{M}$-quasiconvex subset is a subgroup, then it is itself biautomatic. More precisely, we have the following result.

\begin{thm}[see {\cite[Theorem~3.1 and its proof]{gersten-short}}] \label{thm:qconv-biauto}
Let $(Y,\mathcal{M})$ be a biautomatic structure on a group $G$, and let $H \leq G$ be an $\mathcal{M}$-quasiconvex subgroup. Then there exists a biautomatic structure $(X,\mathcal{L})$ for $H$ and a constant $\mu \geq 0$ such that for any $V \in \mathcal{M}$ with $\overline{V} \in H$, there exists $U \in \mathcal{L}$ with $\overline{U} = \overline{V}$, $|U| = |V|$ and $\left| \overline{U(t)}^{-1} \overline{V(t)} \right|_Y \leq \mu$ for all $t \in \Z_{\geq 0}$. Moreover, if $(Y,\mathcal{M})$ is finite-to-one then so is $(X,\mathcal{L})$. \qed
\end{thm}

We refer to the biautomatic structure $(X,\mathcal{L})$ given by Theorem~\ref{thm:qconv-biauto} as the biautomatic structure \emph{associated} to $(Y,\mathcal{M})$. It follows from Theorem~\ref{thm:qconv-biauto} that the quasiconvexity relation between groups equipped with biautomatic structures is transitive in the sense of Lemma~\ref{lem:qconv-trans} below. This result is straightforward, but we give a proof for completeness.

\begin{lem} \label{lem:qconv-trans}
Let $G$ be a group with a biautomatic structure $(Y,\mathcal{M})$, let $H \leq G$ be an $\mathcal{M}$-quasiconvex subgroup with the associated biautomatic structure $(X,\mathcal{L})$, and let $K \leq H$ be an $\mathcal{L}$-quasiconvex subgroup. Then $K$ is $\mathcal{M}$-quasiconvex in $G$.
\end{lem}

\begin{proof}
It follows by Theorem~\ref{thm:qconv-biauto} that there exists $\mu \geq 0$ such that for any $V \in \mathcal{M}$ with $\overline{V} \in H$, there exists $U \in \mathcal{L}$ with $\overline{U} = \overline{V}$ and $\left| \overline{V(t)}^{-1} \overline{U(t)} \right|_Y \leq \mu$ for all $t \in \Z_{\geq 0}$. Now if in addition we have $\overline{V} \in K$, then $\overline{U} \in K$ and so, by Definition~\ref{defn:qconv}, for each $t \in \Z_{\geq 0}$ there exists $k_t \in K$ such that $\left| \overline{U(t)}^{-1} k_t \right|_X \leq \nu$, where $\nu \geq 0$ is some universal constant. If we set $\delta := \max \{ |\overline{x}|_Y \mid x \in X \}$, we then have
\[
\left| \overline{V(t)}^{-1} k_t \right|_Y \leq \left| \overline{V(t)}^{-1} \overline{U(t)} \right|_Y + \left| \overline{U(t)}^{-1} k_t \right|_Y \leq \mu + \delta \nu,
\]
for all $t$, and so any path in $\Cay(G,Y)$ represented by $V$ whose endpoints belong to $K$ is in the $(\mu+\delta\nu)$-neighbourhood of $K$. Thus $K$ is $\mathcal{M}$-quasiconvex in $G$, as required.
\end{proof}

One of the main sources of $\mathcal{M}$-quasiconvex subgroups in a biautomatic group $G$ are centralisers of finite subsets, as per the following result.

\begin{prop}[{\cite[Proposition~4.3]{gersten-short}}] \label{prop:centr-qconv}
Let $(Y,\mathcal{M})$ be a biautomatic structure on a group $G$, and let $S \subseteq G$ be a finite subset. Then $C_G(S)$ is $\mathcal{M}$-quasiconvex. \qed
\end{prop}

\subsection{Polyhedral functions} \label{ssec:poly}

Finally, we introduce polyhedral functions, which we will use to approximate lengths of words in $\mathcal{L}$, where $(X,\mathcal{L})$ is a biautomatic structure for $\Z^n$.

\begin{defn} \label{defn:polyhedral}
Given a finite subset $Z = \{ \mathbf{z}_1,\ldots,\mathbf{z}_\alpha\} \subseteq \R^n$, a \emph{polyhedral cone over $Z$} is the set $C(Z) = \left\{ \sum_{j=1}^\alpha \mu_j \mathbf{z}_j \,\middle|\, \mu_1,\ldots,\mu_\alpha \geq 0 \right\} \subseteq \R^n$. Given a polyhedral cone $C = C(Z)$ and $\mathbf{y} \in \R^n$ such that $\langle \mathbf{z},\mathbf{y} \rangle > 0$ for all $\mathbf{z} \in Z$ (equivalently, $\langle \mathbf{z},\mathbf{y} \rangle > 0$ for all $\mathbf{z} \in C \setminus \{\mathbf{0}\}$), we define a function $f_{C,\mathbf{y}}\colon \R^n \to \R$ by
\[
f_{C,y}(\mathbf{v}) = \begin{cases} \langle \mathbf{v},\mathbf{y} \rangle & \text{if } \mathbf{v} \in C, \\ 0 & \text{otherwise}. \end{cases}
\]

We say $f\colon \R^n \to \R$ is a \emph{polyhedral function} if there exists a finite collection of functions $f_{C_1,\mathbf{y}_1},\ldots,f_{C_m,\mathbf{y}_m}\colon \R^n \to \R$ as above such that
\begin{enumerate}
\item \label{it:polyhedral-cover} $\R^n = \bigcup_{j=1}^m C_j$;
\item \label{it:polyhedral-compat} if $\mathbf{v} \in C_j \cap C_k$ then $\langle \mathbf{v},\mathbf{y}_j \rangle = \langle \mathbf{v},\mathbf{y}_k \rangle$; and
\item \label{it:polyhedral-defn} $f(\mathbf{v}) = \max \{ f_{C_j,\mathbf{y}_j}(\mathbf{v}) \mid 1 \leq j \leq m \}$ for all $\mathbf{v} \in \R^n$.
\end{enumerate}
\end{defn}

Note that if $f$ is a polyhedral function then $f$ is continuous and positively homogeneous: that is, $f(\mu \mathbf{v}) = \mu f(\mathbf{v})$ for all $\mathbf{v} \in \R^n$ and $\mu \geq 0$. Such functions have been studied before, for instance, in \cite{melzer}: in their notation, a polyhedral function is precisely a function that belongs to $\mathscr{P}_+(E_n)$ and is positive (that is, $f(\mathbf{v}) \geq 0$ for all $\mathbf{v}$, with equality if and only if $\mathbf{v} = \mathbf{0}$).

In this paper we will use the following well-known alternative characterisation of polyhedral cones. Here, a \emph{linear halfspace} is a subset of $\R^n$ of the form $\{ \mathbf{v} \in \R^n \mid \langle \mathbf{v},\mathbf{w} \rangle \geq 0 \}$ for some $\mathbf{w} \in \R^n \setminus \{\mathbf{0}\}$.

\begin{thm}[J.~Farkas, H.~Minkowski and H.~Weyl; see {\cite[Theorems~6~\&~7]{kaibel}}] \label{thm:coneequiv}
Let $C \subseteq \R^n$ be a subset. Then the following are equivalent:
\begin{enumerate}
\item $C$ is a polyhedral cone over some finite subset $Z \subseteq \R^n$;
\item there exist linear halfspaces $K_1,\ldots,K_\beta \subseteq \R^n$ such that $C = \bigcap_{j=1}^\beta K_j$. \qed
\end{enumerate}
\end{thm}

\begin{ex} \label{ex:intro}
An example of a polyhedral function is depicted in Figure~\ref{fig:poly}, where the dotted line denotes $f^{-1}(c)$ for some constant $c > 0$. Here we set $C_j = C\left( \{ \mathbf{z}_{j,1},\mathbf{z}_{j,2} \} \right)$ for $1 \leq j \leq 6$, where
\begin{align*}
\mathbf{z}_{1,1} &= \textstyle\left( \frac{1}{4},\frac{1}{2} \right),
    & \mathbf{z}_{1,2} &= \textstyle\left( 0,\frac{1}{2} \right),
    & \mathbf{y}_1 &= (0,2), \\
\mathbf{z}_{2,1} = \mathbf{z}_{3,1} &= \textstyle\left( \frac{1}{4},\frac{1}{2} \right),
    & \mathbf{z}_{2,2} = \mathbf{z}_{3,2} &= \textstyle\left( \frac{1}{4},0 \right),
    & \mathbf{y}_2 = \mathbf{y}_3 &= (4,0), \\
\mathbf{z}_{4,1} &= \textstyle\left( \frac{1}{4},0 \right),
    & \mathbf{z}_{4,2} &= \textstyle\left( 0,-\frac{1}{2} \right),
    & \mathbf{y}_4 &= (4,-2), \\
\mathbf{z}_{5,1} &= \textstyle\left( -\frac{1}{2},0 \right),
    & \mathbf{z}_{5,2} &= \textstyle\left( 0,-\frac{1}{2} \right),
    & \mathbf{y}_5 &= (-2,-2), \\
\mathbf{z}_{6,1} &= \textstyle\left( 0,\frac{1}{2} \right),
    & \mathbf{z}_{6,2} &= \textstyle\left( -\frac{1}{2},0 \right),
    & \mathbf{y}_6 &= (-2,2).
\end{align*}
It is easy to check that the conditions \ref{it:polyhedral-cover} and \ref{it:polyhedral-compat} in Definition~\ref{defn:polyhedral} are satisfied.
\end{ex}

% If anyone is reading this: you might wish to recompile the document yourself to make the picture nicer, as apparently arXiv doesn't seem to support opacity gradients.

\begin{figure}[ht]
\begin{tikzpicture}
\newcommand{\xmax}{5.3}
\fill [violet!20] (0,\xmax) -- (0,0) -- (\xmax/2,\xmax);
\fill [orange!50!red!20] (\xmax/2,\xmax) -- (0,0) -- (\xmax,0) -- (\xmax,\xmax);
\fill [yellow!30] (0,0) rectangle (\xmax,-\xmax);
\fill [green!25] (0,0) rectangle (-\xmax,-\xmax);
\fill [blue!50!green!20] (0,0) rectangle (-\xmax,\xmax);
\draw [violet,->,ultra thick] (0,4.75) node [below right] {$C_1$} (0,0) -- (0,2) node [below right] {$\mathbf{y}_1$};
\draw [orange!50!red,->,ultra thick] (4.75,4.75) node [below left] {$C_2=C_3$} (0,0) -- (4,0) node [above left] {$\mathbf{y}_2=\mathbf{y}_3$};
\draw [yellow!60!black,->,ultra thick] (4.75,-4.75) node [above left] {$C_4$} (0,0) -- (4,-2) node [below left] {$\mathbf{y}_4$};
\draw [green!70!black,->,ultra thick] (-4.75,-4.75) node [above right] {$C_5$} (0,0) -- (-2,-2) node [left] {$\mathbf{y}_5$};
\draw [blue!50!green!70!black,->,ultra thick] (-4.75,4.75) node [below right] {$C_6$} (0,0) -- (-2,2) node [above] {$\mathbf{y}_6$};
\draw [dotted,ultra thick] (0,3) -- (1.5,3) -- (1.5,0) -- (0,-3) -- (-3,0) -- (0,3) -- cycle;
\draw [black,very thin] (0,0) -- (0,\xmax) (1,0) -- (1,\xmax) (2,2) -- (2,\xmax) (3,4) -- (3,\xmax) (0,0) -- (\xmax,0) (1,1) -- (\xmax,1) (1,2) -- (\xmax,2) (2,3) -- (\xmax,3) (2,4) -- (\xmax,4) (3,5) -- (\xmax,5) (0,0) -- (0,-\xmax) (1,0) -- (1,-\xmax) (2,0) -- (2,-\xmax) (3,0) -- (3,-\xmax) (4,0) -- (4,-\xmax) (5,0) -- (5,-\xmax) (-1,0) -- (-1,-\xmax) (-2,0) -- (-2,-\xmax) (-3,0) -- (-3,-\xmax) (-4,0) -- (-4,-\xmax) (-5,0) -- (-5,-\xmax) (0,0) -- (-\xmax,0) (0,1) -- (-\xmax,1) (0,2) -- (-\xmax,2) (0,3) -- (-\xmax,3) (0,4) -- (-\xmax,4) (0,5) -- (-\xmax,5);
\fill [white,path fading=fade bottom] (-\xmax-0.001,\xmax+0.001) rectangle (\xmax+0.001,4.7);
\fill [white,path fading=fade top] (-\xmax-0.001,-\xmax-0.001) rectangle (\xmax+0.001,-4.7);
\fill [white,path fading=fade right] (-\xmax-0.001,-\xmax-0.001) rectangle (-4.7,\xmax+0.001);
\fill [white,path fading=fade left] (\xmax+0.001,-\xmax-0.001) rectangle (4.7,\xmax+0.001);
\end{tikzpicture}
\caption{A representation of a polyhedral function $f\colon \R^2 \to \R$. See Examples \ref{ex:intro}, \ref{ex:section} and \ref{ex:biauto} for details.%
}
\label{fig:poly}
\end{figure}

\section{Geometry of polyhedral functions} \label{sec:geompoly}

Our first result on polyhedral functions says that the restriction of a polyhedral function to a linear subspace is also polyhedral.

\begin{lem} \label{lem:polyhedralsubsp}
Let $\theta\colon \R^n \to \R^N$ be a linear isometric embedding, and let $f\colon \R^N \to \R$ be a polyhedral function. Then $f \circ \theta\colon \R^n \to \R$ is a polyhedral function.
\end{lem}

\begin{proof}
Since any linear isometric embedding can be expressed as a composite of linear isometric embeddings $\theta'\colon \R^{n'} \to \R^{n'+1}$, it is enough to consider the case $N = n+1$. In particular, under this assumption there exists $\mathbf{u} \in \R^{n+1}$ with $\langle \mathbf{u},\mathbf{u} \rangle = 1$ such that $\theta(\R^n) = \{ \mathbf{v} \in \R^{n+1} \mid \langle \mathbf{v},\mathbf{u} \rangle = 0 \}$.

We first show that if $C \subseteq \R^{n+1}$ is a polyhedral cone then so is $\theta^{-1}(C) \subseteq \R^n$. Indeed, by Theorem~\ref{thm:coneequiv}, in that case we have $C = \bigcap_{j=1}^\beta K_j$ for some linear halfspaces $K_1,\ldots,K_\beta \subseteq \R^{n+1}$; for $1 \leq j \leq \beta$, let $\mathbf{w}_j \in \R^{n+1} \setminus \{\mathbf{0}\}$ be such that $K_j = \{ \mathbf{v} \in \R^{n+1} \mid \langle \mathbf{v},\mathbf{w}_j \rangle \geq 0 \}$. We then have $\mathbf{w}_j - \langle \mathbf{u},\mathbf{w}_j \rangle \mathbf{u} \in \theta(\R^n)$ and so we may set $\mathbf{w}_j' := \theta^{-1}(\mathbf{w}_j - \langle \mathbf{u},\mathbf{w}_j \rangle \mathbf{u})$; moreover, we set $K_j' := \{ \mathbf{v}' \in \R^n \mid \langle \mathbf{v}',\mathbf{w}_j' \rangle \geq 0 \}$, so that either $K_j' = \R^n$ or $K_j' \subseteq \R^n$ is a linear halfspace. Given any $\mathbf{v}' \in \R^n$ we have $\langle \theta(\mathbf{v}'),\mathbf{u} \rangle = 0$, and so
\begin{align*}
\mathbf{v}' \in K_j' &\iff \langle \mathbf{v}',\mathbf{w}_j' \rangle \geq 0 \iff \big\langle \theta(\mathbf{v}'), \mathbf{w}_j-\langle \mathbf{u},\mathbf{w}_j\rangle \mathbf{u} \big\rangle \geq 0 \iff \langle \theta(\mathbf{v}'), \mathbf{w}_j \rangle \geq 0 \\ &\iff \theta(\mathbf{v}') \in K_j,
\end{align*}
implying that $K_j' = \theta^{-1}(K_j)$. It follows that $\theta^{-1}(C) = \bigcap_{j=1}^\beta K_j'$ and so (by Theorem~\ref{thm:coneequiv}) $\theta^{-1}(C) \subseteq \R^n$ is a polyhedral cone, as claimed.

Now since $f\colon \R^{n+1} \to \R$ is polyhedral, there exist functions $f_{C_1,\mathbf{y}_1},\ldots,f_{C_m,\mathbf{y}_m}\colon \R^{n+1} \to \R$ as in Definition~\ref{defn:polyhedral}. For $1 \leq j \leq m$, we set $\mathbf{y}_j' := \theta^{-1}(\mathbf{y}_j - \langle \mathbf{u},\mathbf{y}_j \rangle \mathbf{u})$ and $C_j' := \theta^{-1}(C_j)$. Given any $\mathbf{z}' \in C_j'$ we have $\theta(\mathbf{z}') \in C_j$ and $\langle \theta(\mathbf{z}'),\mathbf{u} \rangle = 0$, and so
\[
\langle \mathbf{z}',\mathbf{y}_j' \rangle = \big\langle \theta(\mathbf{z}'),\mathbf{y}_j-\langle \mathbf{u},\mathbf{y}_j \rangle \mathbf{u} \big\rangle = \langle \theta(\mathbf{z}'),\mathbf{y}_j \rangle > 0,
\]
implying that we may define a function $f_{C_j',\mathbf{y}_j'}\colon \R^n \to \R$ as in Definition~\ref{defn:polyhedral}.

We claim that $f \circ \theta$ may be constructed from the functions $f_{C_j',\mathbf{y}_j'}$ as in Definition~\ref{defn:polyhedral}. Note first that we have
\[
\R^n = \theta^{-1}(\R^{n+1}) = \theta^{-1}\Bigg( \bigcup_{j=1}^m C_j \Bigg) = \bigcup_{j=1}^m \theta^{-1}(C_j) = \bigcup_{j=1}^m C_j',
\]
showing condition~\ref{it:polyhedral-cover}. Moreover, if $\mathbf{v}' \in C_j' \cap C_k'$ then we have $\theta(\mathbf{v}') \in C_j \cap C_k$ and so
\begin{align*}
\langle \mathbf{v}',\mathbf{y}_j' \rangle &= \big\langle \theta(\mathbf{v}'),\mathbf{y}_j-\langle \mathbf{u},\mathbf{y}_j \rangle \mathbf{u} \big\rangle = \langle \theta(\mathbf{v}'),\mathbf{y}_j \rangle \\ &= \langle \theta(\mathbf{v}'),\mathbf{y}_k \rangle = \big\langle \theta(\mathbf{v}'),\mathbf{y}_k-\langle \mathbf{u},\mathbf{y}_k \rangle \mathbf{u} \big\rangle = \langle \mathbf{v}',\mathbf{y}_k' \rangle,
\end{align*}
showing condition~\ref{it:polyhedral-compat}. Finally, for any $\mathbf{v}' \in \R^n$ we have $f_{C_j,\mathbf{y}_j} \circ \theta(\mathbf{v}') = 0 = f_{C_j',\mathbf{y}_j'}(\mathbf{v}')$ if $\mathbf{v}' \notin C_j'$, and
\[
f_{C_j,\mathbf{y}_j} \circ \theta(\mathbf{v}') = \langle \theta(\mathbf{v}'),\mathbf{y}_j \rangle = \big\langle \theta(\mathbf{v}'), \mathbf{y}_j-\langle \mathbf{u},\mathbf{y}_j \rangle \mathbf{u} \big\rangle = \langle \mathbf{v}',\mathbf{y}_j' \rangle = f_{C_j',\mathbf{y}_j'}(\mathbf{v}')
\]
if $\mathbf{v}' \in C_j'$, implying that $f_{C_j,\mathbf{y}_j} \circ \theta = f_{C_j',\mathbf{y}_j'}$. Thus $f \circ \theta(\mathbf{v}') = \max \{ f_{C_j',\mathbf{y}_j'}(\mathbf{v}') \mid 1 \leq j \leq m \}$ for all $\mathbf{v}' \in \R^n$, showing condition~\ref{it:polyhedral-defn}. It follows that $f \circ \theta$ is polyhedral, as required.
\end{proof}

We now turn our attention to group actions that preserve the values of a polyhedral function. In particular, in Proposition~\ref{prop:finite} we show that a polyhedral function cannot be $G$-invariant for any infinite subgroup $G \leq GL_n(\R)$. In order to prove this, we will use the following result.

\begin{lem} \label{lem:fixhplanes}
Let $\{ \mathbf{y}_j \mid j \in \mathcal{I} \}$ be a spanning set of $\R^n$, and let $A \in GL_n(\R)$ be such that $A(K_j) = K_j$, where $K_j := \{ \mathbf{v} \in \R^n \mid \langle \mathbf{v},\mathbf{y}_j \rangle = 1 \}$, for each $j \in \mathcal{I}$. Then $A = I_n$.
\end{lem}

\begin{proof}
Since the $\mathbf{y}_j$ span $\R^n$, some subset $\{ \mathbf{y}_{j_1},\ldots,\mathbf{y}_{j_n} \}$ of the $\mathbf{y}_j$ form a basis for $\R^n$. Now for each $k \in \{ 1,\ldots,n \}$, there exists a unique $\mathbf{z}_k \in \R^n$ such that $\langle \mathbf{z}_k,\mathbf{y}_{j_k} \rangle = 1$ and $\langle \mathbf{z}_k,\mathbf{y}_{j_\ell} \rangle = 0$ for all $\ell \neq k$; moreover, it is easy to see that $\{ \mathbf{z}_1,\ldots,\mathbf{z}_n \}$ is a basis for $\R^n$---the basis \emph{dual} to $\{ \mathbf{y}_{j_1},\ldots,\mathbf{y}_{j_n} \}$.

Now let $\mathbf{z} = \sum_{k=1}^n \mathbf{z}_k$. It is then easy to verify that $\bigcap_{\ell=1}^n K_{j_\ell} = \{\mathbf{z}\}$, and that $\bigcap_{1 \leq \ell \leq n, \ell \neq k} K_{j_\ell} = \{ \mathbf{z}+\lambda \mathbf{z}_k \mid \lambda \in \R \}$ for all $k$. As $A(K_{j_\ell}) = K_{j_\ell}$ for all $\ell$, it thus follows that $A\mathbf{z} = \mathbf{z}$ and $\{ A\mathbf{z}+\lambda A\mathbf{z}_k \mid \lambda \in \R \} = \{ \mathbf{z}+\lambda \mathbf{z}_k \mid \lambda \in \R \}$ for all $k$. This implies that for each $k$, there exists $\lambda_k \in \R \setminus \{0\}$ such that $A\mathbf{z}_k = \lambda_k \mathbf{z}_k$. But then we have
\[
\sum_{k=1}^n \mathbf{z}_k = \mathbf{z} = A\mathbf{z} = \sum_{k=1}^n A\mathbf{z}_k = \sum_{k=1}^n \lambda_k \mathbf{z}_k.
\]
As the $\mathbf{z}_k$ are linearly independent, it follows that $\lambda_k = 1$, and so $A\mathbf{z}_k = \mathbf{z}_k$, for all $k$. As the $\mathbf{z}_k$ span $\R^n$, it thus follows that $A = I_n$, as required.
\end{proof}

\begin{prop} \label{prop:finite}
Let $f\colon \R^n \to \R$ be a polyhedral function, and let $G \leq GL_n(\R)$ be a subgroup. If $f(\mathbf{v}) = f(A\mathbf{v})$ for all $\mathbf{v} \in \R^n$ and all $A \in G$, then $G$ is finite.
\end{prop}

\begin{proof}
We use the notation of Definition~\ref{defn:polyhedral}. In what follows, an \emph{affine hyperplane} is a subset of $\R^n$ of the form $\{ \mathbf{v} \in \R^n \mid \langle \mathbf{v},\mathbf{y} \rangle = c \}$ for some $\mathbf{y} \in \R^n \setminus \{\mathbf{0}\}$ and $c \in \R$.

Let $\mathcal{I} \subseteq \{1,\ldots,m\}$ be the set of all $j$ such that $C_j \subseteq \R^n$ has non-empty interior (if $C_j = C(Z_j)$, this is equivalent to saying that $Z_j$ spans $\R^n$). Then, by condition~\ref{it:polyhedral-cover} in Definition~\ref{defn:polyhedral}, we have $\R^n \setminus \bigcup_{j \in \mathcal{I}} C_j \subseteq  \bigcup_{j \notin \mathcal{I}} C_j$, where the left hand-side is open and the right hand side is a finite union of convex subsets with empty interior; this implies that the left hand side is actually empty, and so $\R^n = \bigcup_{j \in \mathcal{I}} C_j$.

Therefore, conditions \ref{it:polyhedral-compat} and \ref{it:polyhedral-defn} imply that $f^{-1}(1) \subseteq \bigcup_{j \in \mathcal{I}} K_j$, where we set $K_j := \{ \mathbf{v} \in \R^n \mid \langle \mathbf{v},\mathbf{y}_j \rangle = 1 \}$. Moreover, for each $j \in \mathcal{I}$, the set $K_j \cap f^{-1}(1)$ has non-empty interior in $K_j$. The converse is also true: if $K$ is an affine hyperplane not equal to any $K_j$ for $j \in \mathcal{I}$, then we have $K \cap f^{-1}(1) \subseteq K \cap \bigcup_{j \in \mathcal{I}} K_j = \bigcup_{j \in \mathcal{I}} (K \cap K_j)$, and the latter is a finite union of $(n-2)$-dimensional affine subspaces which therefore must have empty interior in $K$. Thus, the set $\mathcal{K} := \{ K_j \mid j \in \mathcal{I} \}$  of affine hyperplanes consists of precisely those $K$ for which $K \cap f^{-1}(1)$ has non-empty interior in $K$.

Now the group $G$ has a canonical action on the set of all affine hyperplanes in $\R^n$. Moreover, if $A \in G$ then by the assumptions $A\left( f^{-1}(1) \right) = f^{-1}(1)$, implying that if $K \cap f^{-1}(1)$ has non-empty interior in an affine hyperplane $K$, then $A(K) \cap f^{-1}(1)$ has non-empty interior in $A(K)$. Thus the set $\mathcal{K}$ is $G$-invariant, and so we have homomorphism $\Phi\colon G \to \Sym(\mathcal{K})$. As $\mathcal{K}$ is finite, it is then enough to show that $\Phi$ is injective.

We first claim that the set $\{ \mathbf{y}_j \mid j \in \mathcal{I} \}$ spans $\R^n$. Suppose for contradiction that the set $\{ \mathbf{y}_j \mid j \in \mathcal{I} \}$ spans a proper subspace of $\R^n$. Then there exists $\mathbf{u} \in \R^n \setminus \{\mathbf{0}\}$ such that $\langle \mathbf{u},\mathbf{y}_j \rangle = 0$ for all $j \in \mathcal{I}$. But, as shown above, we have $\R^n = \bigcup_{j \in \mathcal{I}} C_j$, and so in that case $\mathbf{u} \in C_k$ for some $k \in \mathcal{I}$. This is impossible, as we have $\langle \mathbf{z},\mathbf{y}_k \rangle > 0$ for all $\mathbf{z} \in C_k \setminus \{\mathbf{0}\}$. Thus indeed $\{ \mathbf{y}_j \mid j \in \mathcal{I} \}$ spans $\R^n$, as claimed.

Now let $A \in \ker(\Phi)$, so that $A(K_j) = K_j$ for all $j \in \mathcal{I}$. It then follows from Lemma~\ref{lem:fixhplanes} that $A = I_n$. Therefore, $\Phi\colon G \to \Sym(\mathcal{K})$ is injective and so $G$ is finite, as required.
\end{proof}

\begin{ex} \label{ex:section}
Let $f$ be the polyhedral function depicted in Figure~\ref{fig:poly}, and $\theta\colon \R \to \R^2$ be an isometric embedding whose image is the diagonal $\{ (v,v) \mid v \in \R \} \subset \R^2$. We may then check that $f\circ \theta(v) = 2\sqrt{2} \left| v \right|$ for all $v \in \R$, and so $f \circ \theta$ is polyhedral, as per Lemma~\ref{lem:polyhedralsubsp}: in the notation of Definition~\ref{defn:polyhedral} we could take
\begin{align*}
C_1 &:= C(\{1\}) = [0,\infty), & \mathbf{y}_1 &:= 2\sqrt{2}, \\ C_2 &:= C(\{-1\}) = (-\infty,0], & \mathbf{y}_2 &:= -2\sqrt{2}
\end{align*}
to define $f \circ \theta$.

A straightforward calculation shows that for any non-identity matrix $A \in GL_2(\R)$ there exists $\mathbf{v} \in \R^2$ such that $f(A\mathbf{v}) \neq f(\mathbf{v})$, and so if $G \leq GL_2(\R)$ is such that $f(\mathbf{v}) = f(A\mathbf{v})$ for all $\mathbf{v} \in \R^2$ and $A \in G$, then $G$ must be trivial. On the other hand, $f \circ \theta(-v) = f \circ \theta(v)$ for all $v \in \R$. Nevertheless, if $G \leq GL_1(\R)$ is a subgroup such that $f \circ \theta(v) = f \circ \theta(Av)$ for all $v \in \R$ and $A \in G$, then we must have $G \leq \left\{ \begin{pmatrix}1\end{pmatrix},\begin{pmatrix}-1\end{pmatrix} \right\}$, and so $G$ is still finite, as per Proposition~\ref{prop:finite}.
\end{ex}

\section{A polyhedral function associated to a biautomatic structure on \texorpdfstring{$\Z^n$}{Z{\textasciicircum}n}} \label{sec:biauto-poly}

In this section, we associate to any biautomatic structure $(X,\mathcal{L})$ on $\Z^N$ a polyhedral function. Our aim is to do this in such a way that given an element $\mathbf{v} \in \Z^N$ represented by a word $U \in \mathcal{L}$, the length $|U|$ of $U$ can be roughly approximated by $f(\mathbf{v})$: see Proposition~\ref{prop:biauto-poly}.

We first need the following auxiliary result.

\begin{lem} \label{lem:indep}
Let $(X,\mathcal{L})$ be a finite-to-one biautomatic structure on $\Z^N$, and suppose that there exist $U_0,V_1,U_1,\ldots,U_\alpha,V_\alpha \in X^*$ such that $U_0 \cdot V_1^* \cdot U_1 \cdot {} \cdots {} \cdot V_\alpha^* \cdot U_\alpha \subseteq \mathcal{L}$. Then the set $\{ \mathbf{z}_1,\ldots,\mathbf{z}_\alpha \} \subseteq \Q^N$, where $\mathbf{z}_j = \overline{V_j} / |V_j|$, is linearly independent. [For the avoidance of doubt, we do allow having $\mathbf{z}_j = \mathbf{z}_{j'}$ for $j \neq j'$---that is, we claim that the $\mathbf{z}_j$ become linearly independent after deleting repetitions.]
\end{lem}

\begin{proof}
Suppose for contradiction that the set $\{ \mathbf{z}_1,\ldots,\mathbf{z}_\alpha \}$ is not independent. Then there exist $\mu_1,\ldots,\mu_\alpha \in \R$, not all zero, such that if $\mathbf{z}_j = \mathbf{z}_{j'}$ for some $1 \leq j < j' \leq \alpha$ then $\mu_{j'} = 0$, and such that
\[
\sum_{j=1}^\alpha \mu_j \mathbf{z}_j = 0.
\]
Since $\mathbf{z}_j \in \Q^N$ we can also choose $\mu_j \in \Q$ for all $j$. Without loss of generality, assume also that $\mu_j > 0$ for some $j$, and (by rescaling the $\mu_j$ if necessary) that $\frac{\mu_j}{|V_j|} \in \Z$ for all $j$. We now consider two cases---depending on whether or not $\mu_j < 0$ for some $j$---obtaining a contradiction in each.

Suppose first that $\mu_j \geq 0$ for all $j$. It then follows that for each $\beta \in \Z_{\geq 0}$, the word
\[
U_0 V_1^{\beta\mu_1/|V_1|} U_1 \cdots V_\alpha^{\beta\mu_\alpha/|V_\alpha|} U_\alpha \in \mathcal{L}
\]
represents the element $\sum_{j=0}^\alpha \overline{U_j} \in \Z^N$. As $\mu_j > 0$ for some $j$, this gives infinitely many words representing a single element of $\mathcal{L}$, contradicting the fact that $(X,\mathcal{L})$ is finite-to-one.

Suppose now that, on the contrary, $\mu_j < 0$ for some $j$. Let $\mathcal{I}_+ = \{ j \mid \mu_j > 0 \}$ and $\mathcal{I}_- = \{ j \mid \mu_j < 0 \}$; by the assumptions, both $\mathcal{I}_+$ and $\mathcal{I}_-$ are non-empty. Then for each $\beta \in \Z_{\geq 0}$, the words
\begin{align*}
W_\beta^+ &:= U_0 V_1^{\beta\mu_1^+/|V_1|} U_1 \cdots V_\alpha^{\beta\mu_\alpha^+/|V_\alpha|} U_\alpha \in \mathcal{L}
\shortintertext{and}
W_\beta^- &:= U_0 V_1^{\beta\mu_1^-/|V_1|} U_1 \cdots V_\alpha^{\beta\mu_\alpha^-/|V_\alpha|} U_\alpha \in \mathcal{L},
\end{align*}
where, for $\varepsilon\in\{\pm\}$, $\mu_j^\varepsilon = |\mu_j|$ if $j \in \mathcal{I}_\varepsilon$ and $\mu_j^\varepsilon = 0$ otherwise, represent the same element of $\Z^N$. This means that $W_\beta^+$ and $W_\beta^-$ satisfy the fellow traveller property (see Definition~\ref{defn:biauto}) for some constant $\lambda \geq 0$ independent of $\beta$.

Now let $j_+ = \min \mathcal{I}_+$ and $j_- = \min \mathcal{I}_-$. Then the prefixes of $W_\beta^+$ and $W_\beta^-$ of length $t = t(\beta) = \beta \mu + \sum_{j=0}^\alpha |U_j|$, where $\mu = \min \{\mu_{j_+},-\mu_{j_-}\}%\left\{ \frac{\mu_{j_+}}{|V_{j_+}|}, \frac{-\mu_{j_-}}{|V_{j_-}|} \right\}
$, are
\begin{align*}
W_\beta^+(t) &:= U_0 \cdots U_{j_+-1} V_{j_+}^{\left\lfloor \beta\mu/|V_{j_+}| \right\rfloor} Y_\beta^+
\shortintertext{and}
W_\beta^-(t) &:= U_0 \cdots U_{j_--1} V_{j_-}^{\left\lfloor \beta\mu/|V_{j_-}| \right\rfloor} Y_\beta^-
\end{align*}
respectively, where $Y_\beta^+$ and $Y_\beta^-$ are some words of length $\leq |V_{j_+}| + |V_{j_-}| + \sum_{j=0}^\alpha |U_j|$. But this means that $\overline{W_\beta^+(t)} - \overline{W_\beta^-(t)}$ is bounded distance away from the point $\frac{\beta\mu}{|V_{j_+}|}\overline{V_{j_+}} - \frac{\beta\mu}{|V_{j_-}|}\overline{V_{j_-}} = \beta\mu(\mathbf{z}_{j_-}-\mathbf{z}_{j_+}) \in \Q^N$ with respect to any fixed norm on $\Q^N$. As by assumptions $\mu > 0$ and $\mathbf{z}_{j_+} \neq \mathbf{z}_{j_-}$, it follows that $\overline{W_\beta^+(t)} - \overline{W_\beta^-(t)}$ is arbitrarily far from the origin for large $\beta$, contradicting the fellow traveller property.\end{proof}

\begin{prop} \label{prop:biauto-poly}
Let $(X,\mathcal{L})$ be a finite-to-one biautomatic structure on $\Z^N$. Then there exist a polyhedral function $f\colon \R^N \to \R$ and a constant $\xi \geq 0$ such that
\[
f(\overline{U}) - \xi \leq \left| U \right| \leq f(\overline{U}) + \xi
\]
for all $U \in \mathcal{L}$.
\end{prop}

\begin{proof}
As $(X,\mathcal{L})$ is finite-to-one and as $\Z^N$ has polynomial growth, it follows that $\mathcal{L}$ has polynomial growth as well---that is, there exists a polynomial $g(x)$ such that for each $n \geq 0$ there are at most $g(n)$ words $U \in \mathcal{L}$ with $|U| \leq n$. It follows from \cite[Proposition~1.3.8]{epstein} that $\mathcal{L}$ is \emph{simply starred}: that is, there exist integers $\alpha_1,\ldots,\alpha_m \in \Z_{\geq 0}$ and words $U_{j,0},V_{j,1},U_{j,1},\ldots,V_{j,\alpha_j},U_{j,\alpha_j} \in X^*$ (for each $j \in \{1,\ldots,m\}$) such that
\begin{equation} \label{eq:langunion}
\mathcal{L} = \bigcup_{j=1}^m U_{j,0} \cdot (V_{j,1})^* \cdot U_{j,1} \cdot {} \cdots {} \cdot (V_{j,\alpha_j})^* \cdot U_{j,\alpha_j}.
\end{equation}

By Lemma~\ref{lem:indep}, for each $j$ the set $\{ \mathbf{z}_{j,1},\ldots,\mathbf{z}_{j,\alpha_j} \} \subseteq \R^N$, where $\mathbf{z}_{j,k} = \frac{\overline{V_{j,k}}}{|V_{j,k}|}$, is linearly independent. Thus, if $H_{j,k} \subseteq \R^N$ is the affine hyperplane $\{ \mathbf{v} \in \R^N \mid \langle \mathbf{z}_{j,k},\mathbf{v} \rangle = 1 \}$, then for $1 \leq j \leq m$ the intersection $\bigcap_{k=1}^{\alpha_j} H_{j,k}$ is non-empty, and so contains a point $\mathbf{y}_j \in \bigcap_{k=1}^{\alpha_j} H_{j,k}$. For each $j \in \{ 1,\ldots,m \}$, let $C_j \subseteq \R^N$ be the polyhedral cone over $\{ \mathbf{z}_{j,1},\ldots,\mathbf{z}_{j,\alpha_j} \}$; notice that we have $\langle \mathbf{z}_{j,k}, \mathbf{y}_j \rangle = 1 > 0$ for all $k$ by construction, and so we may define a function $f_{C_j,\mathbf{y}_j}\colon \R^N \to \R$ as in Definition~\ref{defn:polyhedral}. We then define $f\colon \R^N \to \R$ by setting
\[
f(\mathbf{v}) := \max \{ f_{C_j,\mathbf{y}_j}(\mathbf{v}) \mid 1 \leq j \leq m \}.
\]

We claim the following.
\begin{lem} \label{lem:polyhedral}
$f$ is a polyhedral function; in particular, in the above notation, conditions \ref{it:polyhedral-cover}, \ref{it:polyhedral-compat} and \ref{it:polyhedral-defn} in Definition~\ref{defn:polyhedral} are satisfied.
\end{lem}

We postpone the proof of Lemma~\ref{lem:polyhedral} until later, and finish the proof of Proposition~\ref{prop:biauto-poly} first.

We now define a few constants, as follows. We set $\delta := \max \left\{ \sum_{k=1}^{\alpha_j} |U_{j,k}| \,\middle|\, 1 \leq j \leq m \right\}$, $\zeta := \max \{ \|\mathbf{y}_j\| \mid 1 \leq j \leq m \}$, and $\eta := \max \{ \| \overline{x} \| \mid x \in X \}$. We then set
\[
\xi := \zeta\eta\delta + \delta.
\]

Notice that if $\mathbf{v},\mathbf{v}' \in C_j$ for some $j$ then condition~\ref{it:polyhedral-compat} in Definition~\ref{defn:polyhedral} implies that $f(\mathbf{v}) = \langle \mathbf{v},\mathbf{y}_j \rangle$ and $f(\mathbf{v}') = \langle \mathbf{v}',\mathbf{y}_j \rangle$, and therefore
\[
|f(\mathbf{v})-f(\mathbf{v}')| = |\langle \mathbf{v}-\mathbf{v}',\mathbf{y}_j \rangle| \leq \left\| \mathbf{v}-\mathbf{v}' \right\| \left\| \mathbf{y}_j \right\| \leq \zeta \left\| \mathbf{v}-\mathbf{v}' \right\|;
\]
this implies, by considering values of $f$ at some intermediate points on the geodesic connecting $\mathbf{v}$ and $\mathbf{v}'$, that in fact $|f(\mathbf{v})-f(\mathbf{v}')| \leq \zeta \left\| \mathbf{v}-\mathbf{v}' \right\|$ for \emph{any} $\mathbf{v},\mathbf{v}' \in \R^N$. In particular, it follows that if $\mathbf{v},\mathbf{v}' \in \Z^N$ then $|f(\mathbf{v})-f(\mathbf{v}')| \leq \zeta\eta|\mathbf{v}-\mathbf{v}'|_X$.

Now let $U \in \mathcal{L}$, so that by \eqref{eq:langunion} we have $U = U_{j,0} V_{j,1}^{\beta_1} U_{j,1} \cdots V_{j,\alpha_j}^{\beta_{\alpha_j}} U_{j,\alpha_j}$ for some $j$ and some $\beta_1,\ldots,\beta_{\alpha_j} \in \Z_{\geq 0}$. We set $\mathbf{v} := \sum_{k=1}^{\alpha_j} \beta_k \overline{V_{j,k}}$, so that we have $\overline{U}-\mathbf{v} = \sum_{k=0}^{\alpha_j} \overline{U_{j,k}}$, implying that
\[
\left| f(\overline{U}) - f(\mathbf{v}) \right| \leq \zeta\eta \left| \sum_{k=0}^{\alpha_j} \overline{U_{j,k}} \right|_X \leq \zeta\eta\sum_{k=0}^{\alpha_j} |U_{j,k}| \leq \zeta\eta\delta.
\]
On the other hand, we have $\mathbf{v} = \sum_{k=1}^{\alpha_j} \beta_k|V_{j,k}| \mathbf{z}_{j,k}$, and we can compute that
\[
f_{C_j,\mathbf{y}_j}(\mathbf{v}) = \langle \mathbf{v},\mathbf{y}_j \rangle = \sum_{k=1}^{\alpha_j} \beta_k|V_{j,k}| \langle \mathbf{z}_{j,k},\mathbf{y}_j \rangle = \sum_{k=1}^{\alpha_j} \beta_k|V_{j,k}|.
\]
Moreover, it follows from the condition~\ref{it:polyhedral-compat} in Definition~\ref{defn:polyhedral} that we have $f(\mathbf{v}) = f_{C_j,\mathbf{y}_j}(\mathbf{v})$, and therefore
\[
|U| = \sum_{k=0}^{\alpha_j} |U_{j,k}| + \sum_{k=1}^{\alpha_j} \beta_k |V_{j,k}| = \sum_{k=0}^{\alpha_j} |U_{j,k}| + f(\mathbf{v}).
\]
We thus have
\[
\left| |U| - f(\overline{U}) \right| \leq \big| |U| - f(\mathbf{v}) \big| + \left| f(\mathbf{v}) - f(\overline{U}) \right| \leq \left| \sum_{k=0}^{\alpha_j} |U_{j,k}| \right| + \zeta\eta\delta \leq \delta + \zeta\eta\delta = \xi,
\]
as required.
\end{proof}

We now prove Lemma~\ref{lem:polyhedral} that was stated in the proof of Proposition~\ref{prop:biauto-poly}.

\begin{proof}[Proof of Lemma~\ref{lem:polyhedral}]
The condition~\ref{it:polyhedral-defn} follows from the construction---so we only need to check \ref{it:polyhedral-cover} and \ref{it:polyhedral-compat}. In what follows, we set $\delta := \max \left\{ \sum_{k=0}^{\alpha_j} |U_{j,k}| \,\middle|\, 1 \leq j \leq m \right\}$. As $(X,\mathcal{L})$ is a finite-to-one biautomatic structure, by Lemma~\ref{lem:simlengths} there exists $\kappa \geq 0$ such that if $U,V \in \mathcal{L}$ are such that $\left| \overline{U}-\overline{V} \right|_X \leq 1$ then $|U|-\kappa \leq |V| \leq |U|+\kappa$.

In order to show \ref{it:polyhedral-cover}, suppose for contradiction that $D := \R^N \setminus \bigcup_{j=1}^m C_j$ is non-empty. Thus, as $D \subseteq \R^N$ is open and $\Q^N \subseteq \R^N$ is dense, we may pick $\mathbf{v} \in D \cap \Q^N$. Since $D$ is invariant under multiplication by $\mu$ for any $\mu > 0$, we may furthermore assume that $\mathbf{v} \in \Z^N$, and that if $\mathbf{w} \in \Z^N$ but $\mathbf{w} \notin D$ then $|\mathbf{v}-\mathbf{w}|_X > \delta$.

Now let $U \in \mathcal{L}$ be a word with $\overline{U} = \mathbf{v}$. By \eqref{eq:langunion}, we have $U = U_{j,0} V_{j,1}^{\beta_1} U_{j,1} \cdots V_{j,\alpha_j}^{\beta_{\alpha_j}} U_{j,\alpha_j}$ for some $j$ and some $\beta_1,\ldots,\beta_{\alpha_j} \in \Z_{\geq 0}$. However, we then have $\mathbf{w} := \sum_{k=1}^{\alpha_j} \beta_k \overline{V_{j,k}} = \sum_{k=1}^{\alpha_j} \beta_k |V_{j,k}| \mathbf{z}_{j,k} \in C_j$, but $|\mathbf{v}-\mathbf{w}|_X \leq \sum_{k=0}^{\alpha_j} |U_{j,k}| \leq \delta$, contradicting the choice of $\mathbf{v}$. Thus $\{ C_j \mid 1 \leq j \leq m \}$ must cover $\R^N$, which shows \ref{it:polyhedral-cover}.

In order to show \ref{it:polyhedral-compat}, let $j,k \in \{ 1,\ldots,m \}$. Since $\mathbf{z}_{j,\ell} \in \Q^N$ and $\mathbf{z}_{k,\ell} \in \Q^N$ for all $\ell$, we may express $C_j \cap C_k$ as the set of solutions of a system of linear inequalities with rational coefficients (see Theorem~\ref{thm:coneequiv}). In particular, it follows that any non-empty open subset of $C_j \cap C_k$ contains a point in $\Q^N$, and so $C_j \cap C_k$ is the closure of $C_j \cap C_k \cap \Q^N$ in $\R^N$. As the functions $\langle {-}, \mathbf{y}_j \rangle$ and $\langle {-}, \mathbf{y}_k \rangle$ are continuous, it is thus enough to verify \ref{it:polyhedral-compat} when $\mathbf{v} \in \Q^N$.

Thus, let $\mathbf{v} \in C_j \cap C_k \cap \Q^N$. Since $C_j$ and $C_k$ are invariant under multiplication by any $\mu > 0$, and since the functions $\langle {-}, \mathbf{y}_j \rangle$ and $\langle {-}, \mathbf{y}_k \rangle$ are linear, we may furthermore assume (after multiplying $\mathbf{v}$ by a positive integer if necessary) that $\mathbf{v} = \sum_{\ell=1}^{\alpha_j} \mu_{j,\ell} \mathbf{z}_{j,\ell} = \sum_{\ell=1}^{\alpha_k} \mu_{k,\ell} \mathbf{z}_{k,\ell}$ with $\mu_{j,\ell}/|V_{j,\ell}| \in \Z_{\geq 0}$ and $\mu_{k,\ell}/|V_{k,\ell}| \in \Z_{\geq 0}$ for all $\ell$. For any $\beta \in \Z_{\geq 0}$, we define the words
\begin{align*}
W_{\beta,j} &= U_{j,0} V_{j,1}^{\beta\mu_{j,1}/|V_{j,1}|} U_{j,1} \cdots V_{j,\alpha_j}^{\beta\mu_{j,\alpha_j}/|V_{j,\alpha_j}|} U_{j,\alpha_j} \in \mathcal{L}
\shortintertext{and}
W_{\beta,k} &= U_{k,0} V_{k,1}^{\beta\mu_{k,1}/|V_{k,1}|} U_{k,1} \cdots V_{k,\alpha_k}^{\beta\mu_{k,\alpha_k}/|V_{k,\alpha_k}|} U_{k,\alpha_k} \in \mathcal{L}.
\end{align*}

We then have
\[
\overline{W_{\beta,j}} = \sum_{\ell=0}^{\alpha_j} \overline{U_{j,\ell}} + \sum_{\ell=1}^{\alpha_j} \frac{\beta\mu_{j,\ell}}{|V_{j,\ell}|} \overline{V_{j,\ell}} = \sum_{\ell=0}^{\alpha_j} \overline{U_{j,\ell}} + \beta \sum_{\ell=1}^{\alpha_j} \mu_{j,\ell} \mathbf{z}_{j,\ell} = \sum_{\ell=0}^{\alpha_j} \overline{U_{j,\ell}} + \beta \mathbf{v},
\]
and similarly $\overline{W_{\beta,k}} = \sum_{\ell=0}^{\alpha_k} \overline{U_{k,\ell}} + \beta \mathbf{v}$. It follows that
\[
\left| \overline{W_{\beta,j}} - \overline{W_{\beta,k}} \right|_X \leq \sum_{\ell=0}^{\alpha_j} |U_{j,\ell}| + \sum_{\ell=0}^{\alpha_k} |U_{k,\ell}| \leq 2\delta,
\]
and so $\Big| |W_{\beta,j}| - |W_{\beta,k}| \Big| \leq 2\delta\kappa$.

On the other hand, we may compute that $|W_{\beta,j}| = \sum_{\ell=0}^{\alpha_j} |U_{j,\ell}| + \beta \sum_{\ell=1}^{\alpha_j} \mu_{j,\ell}$ and $|W_{\beta,k}| = \sum_{\ell=0}^{\alpha_k} |U_{k,\ell}| + \beta \sum_{\ell=1}^{\alpha_k} \mu_{k,\ell}$, implying that
\[
\beta \left| \sum_{\ell=1}^{\alpha_k} \mu_{k,\ell} - \sum_{\ell=1}^{\alpha_j} \mu_{j,\ell} \right| \leq \sum_{\ell=0}^{\alpha_j} |U_{j,\ell}| + \sum_{\ell=0}^{\alpha_k} |U_{k,\ell}| + \Big| |W_{\beta,j}| - |W_{\beta,k}| \Big| \leq 2\delta+2\delta\kappa;
\]
as $\beta$ can be chosen to be arbitrarily large, it follows that $\sum_{\ell=1}^{\alpha_j} \mu_{j,\ell} = \sum_{\ell=1}^{\alpha_k} \mu_{k,\ell}$. Finally, we get
\[
\langle \mathbf{v},\mathbf{y}_j \rangle = \sum_{\ell=1}^{\alpha_j} \mu_{j,\ell} \langle \mathbf{z}_{j,\ell},\mathbf{y}_j \rangle = \sum_{\ell=1}^{\alpha_j} \mu_{j,\ell} = \sum_{\ell=1}^{\alpha_k} \mu_{k,\ell} = \sum_{\ell=1}^{\alpha_k} \mu_{k,\ell} \langle \mathbf{z}_{k,\ell},\mathbf{y}_k \rangle = \langle \mathbf{v},\mathbf{y}_k \rangle.
\]
This proves \ref{it:polyhedral-compat}.
\end{proof}

\begin{rmk} \label{rmk:neumann-shapiro}
Our construction is related to the Neumann--Shapiro triangulation of $\Sb^{N-1}$ associated to a biautomatic structure on $\Z^N$ \cite{neumann-shapiro}. Namely, let $C_1,\ldots,C_m \subseteq \R^N$ be the polyhedral cones constructed in the proof of Proposition~\ref{prop:biauto-poly}. Some subset of these cones---which we get after discarding cones contained in either a proper subspace of $\R^N$ or another cone---is precisely the set of $(N-1)$-simplices in the relevant Neumann--Shapiro triangulation of $\Sb^{N-1}$. We may furthermore order vertices in each of these polyhedral cones, with the ordering induced by the order $\mathbf{z}_{j,1} \prec \cdots \prec \mathbf{z}_{j,\alpha_j}$ on $C_j$, thus recovering the complete structure of the triangulations exhibited in \cite{neumann-shapiro}.

However, we amend this triangulation by constructing a polyhedral function, associating to $(X,\mathcal{L})$ a geometric rather than combinatorial structure. This allows easier treatment of arbitrary subgroups of $\Z^N$. In particular, even though a triangulation of $\Sb^{N-1}$ does not necessarily induce a triangulation on an arbitrary equatorial subsphere $\Sb^{n-1} \subset \Sb^{N-1}$ for $n < N$, composing a polyhedral function with an arbitrary isometric linear inclusion $\R^n \hookrightarrow \R^N$ still yields a polyhedral function by Lemma~\ref{lem:polyhedralsubsp}.
\end{rmk}

\begin{ex} \label{ex:biauto}
Let $X = \{ \varepsilon,x,y,x^{-1},y^{-1} \}$ be the generating set of $\Z^2$ such that $x$, $y$ and $\varepsilon$ map to $(1,0)$, $(0,1)$ and $(0,0)$, respectively. Define the language $\mathcal{L}$ as follows:
\begin{align*}
\mathcal{L} := \left(\varepsilon xy^2\right)^* \left(\varepsilon y\right)^* &\cup \left(\varepsilon xy^2\right)^* \left(\varepsilon^3 x\right)^* \cup \left(\varepsilon xy^2\right)^* y^{-1} \left(\varepsilon^3 x\right)^* \\ &{} \cup \left(\varepsilon^3 x\right)^* \left(\varepsilon y^{-1}\right)^* \cup \left(\varepsilon x^{-1}\right)^* \left(\varepsilon y^{-1}\right)^* \cup \left(\varepsilon y\right)^* \left(\varepsilon x^{-1}\right)^*.
\end{align*}
One may then check that $(X,\mathcal{L})$ is indeed a finite-to-one biautomatic structure for $\Z^2$. Moreover, in this case the polyhedral function constructed in the proof of Proposition~\ref{prop:biauto-poly} is precisely the function $f$ depicted in Figure~\ref{fig:poly}, and one may check that the values of $\mathbf{y}_j$ and $\mathbf{z}_{j,k}$ indicated in Example~\ref{ex:intro} are consistent with the notation used in the proof of Proposition~\ref{prop:biauto-poly} for the language $\mathcal{L}$. The thin black lines in Figure~\ref{fig:poly} represent the paths starting at $(0,0)$ and labelled by words in $\mathcal{L}$.

Proposition~\ref{prop:biauto-poly} then implies that given any $n \in \Z_{\geq 0}$, if $S_n \subset \Z^2$ is the set of elements represented by words in $\mathcal{L}$ of length $n$, then $S_n$ is bounded distance away from an appropriate rescaling of the dotted line in Figure~\ref{fig:poly}, for some bound that is independent of~$n$.
\end{ex}

\section{The proof of Theorem~\ref{thm:main}} \label{sec:proof}

Finally, we use Lemma~\ref{lem:polyhedralsubsp} and Propositions~\ref{prop:finite}~\&~\ref{prop:biauto-poly} to prove Theorem~\ref{thm:main}. The idea of our proof is to embed $H$ (or a finite-index torsion-free subgroup of $H$) into an $\mathcal{M}$-quasiconvex free abelian subgroup $\widehat{H} \leq G$, where $(Y,\mathcal{M})$ is a biautomatic structure of $G$. We then associate a polyhedral function to $\widehat{H}$, as per Proposition~\ref{prop:biauto-poly}, and Lemma~\ref{lem:polyhedralsubsp} allows us to restrict this function to a polyhedral function associated to $H$. The following result then implies that the latter polyhedral function is (in a sense) $K$-invariant, where $K$ is the image of $\Comm_G(H)$ in $\Comm(H)$, and consequently $K$ is finite by Proposition~\ref{prop:finite}.

\begin{prop} \label{prop:polyhedralaction}
Let $G$ be a group with a finite-to-one biautomatic structure $(Y,\mathcal{M})$, let $\widehat{H} \leq G$ be a free abelian $\mathcal{M}$-quasiconvex subgroup of rank $N \geq 0$, and let $\varphi\colon \widehat{H} \to \R^N$ be an embedding with $\varphi(\widehat{H}) = \Z^N$. Then there exists a polyhedral function $f\colon \R^N \to \R$ such that $f \circ \varphi(h) = f \circ \varphi(ghg^{-1})$ for all $g \in G$ and all $h \in \widehat{H} \cap g^{-1}\widehat{H}g$.
\end{prop}

\begin{proof}
As $(Y,\mathcal{M})$ is a finite-to-one biautomatic structure on $G$, by Lemma~\ref{lem:simlengths} there exists a constant $\kappa \geq 0$ such that if $U,V \in \mathcal{M}$ are such that $\overline{U} = \overline{y^{-1}Vy}$ for some $y \in Y$, then $\big| |U|-|V| \big| \leq \kappa$. Since $\widehat{H}$ is $\mathcal{M}$-quasiconvex, by Theorem~\ref{thm:qconv-biauto} there exists a finite-to-one biautomatic structure $(X,\mathcal{L})$ on $\widehat{H}$ such that for any $V \in \mathcal{M}$ with $\overline{V} \in \widehat{H}$, there exists $U \in \mathcal{L}$  with $\overline{U} = \overline{V}$ and $|U| = |V|$. By identifying the subgroup $\widehat{H}$ with $\Z^N$ via $\varphi$, let $f\colon \R^N \to \R$ and $\xi \geq 0$ be the polyhedral function and the constant given by Proposition~\ref{prop:biauto-poly}.

Now let $g \in G$ and let $h \in \widehat{H} \cap g^{-1}\widehat{H}g$, so that $h^\beta,gh^\beta g^{-1} \in \widehat{H}$ for all $\beta \in \Z_{\geq 0}$; without loss of generality, assume that $g \neq 1$. For each $\beta \in \Z_{\geq 0}$, let $U_\beta,V_\beta \in \mathcal{L}$ and $U_\beta',V_\beta' \in \mathcal{M}$ be such that $\overline{U_\beta} = \overline{U_\beta'} = h^\beta$, $\overline{V_\beta} = \overline{V_\beta'} = gh^\beta g^{-1}$, $|U_\beta| = |U_\beta'|$ and $|V_\beta| = |V_\beta'|$. It then follows from Proposition~\ref{prop:biauto-poly} that $\big| f \circ \varphi(h^\beta) - |U_\beta| \big| \leq \xi$ and $\big| f \circ \varphi(gh^\beta g^{-1}) - |V_\beta| \big| \leq \xi$. Moreover, by the choice of the constant $\kappa \geq 0$ above, we have $\big| |U_\beta|-|V_\beta| \big| = \big| |U_\beta'|-|V_\beta'| \big| \leq \kappa|g|_Y$. Finally, it follows from Definition~\ref{defn:polyhedral} that $f(\beta \mathbf{x}) = \beta f(\mathbf{x})$ for all $\beta \in \Z_{\geq 0}$ and $\mathbf{x} \in \R^n$, and therefore
\begin{align*}
&\beta \left| f \circ \varphi(h) - f \circ \varphi(ghg^{-1}) \right| = \left| f \circ \varphi(h^\beta) - f \circ \varphi(gh^\beta g^{-1}) \right| \\ &\qquad\qquad{} \leq \left| f \circ \varphi(h^\beta) - |U_\beta| \right| + \big| |U_\beta|-|V_\beta| \big| + \left| |V_\beta| - f \circ \varphi(gh^\beta g^{-1}) \right| \\ &\qquad\qquad{} \leq 2\xi + \kappa|g|_Y.
\end{align*}
As $\beta \in \Z_{\geq 0}$ was arbitrary, it follows that $f \circ \varphi(h) = f \circ \varphi(ghg^{-1})$, as required.
\end{proof}

\begin{proof}[Proof of Theorem~\ref{thm:main}]
Since $H$ is finitely generated abelian, it has a torsion-free subgroup $H'$ of finite index. By Lemma~\ref{lem:commfi}, it is enough to show that the image of $\Comm_G(H')$ in $\Comm(H')$ is finite. Therefore, we will assume (without loss of generality) that $H$ is torsion-free.

Now let $(Y,\mathcal{M})$ be a biautomatic structure for $G$. Since $H$ is finitely generated, its centraliser $C_G(H)$ is $\mathcal{M}$-quasiconvex by Proposition~\ref{prop:centr-qconv}; let $(Y',\mathcal{M}')$ be the associated biautomatic structure for $C_G(H)$. In particular, $C_G(H)$ is finitely generated and so its centre, $Z(C_G(H))$, is $\mathcal{M}'$-quasiconvex in $C_G(H)$ (again by Proposition~\ref{prop:centr-qconv}), and so $\mathcal{M}$-quasiconvex in $G$ (by Lemma~\ref{lem:qconv-trans}). Let $(Y'',\mathcal{M}'')$ be the associated biautomatic structure for $Z(C_G(H))$.

Now $Z(C_G(H))$ is a finitely generated abelian group containing $H$, and so $Z(C_G(H))/H$ is a finitely generated abelian group. Thus $Z(C_G(H))/H$ has a finite-index torsion-free subgroup, say $\widehat{H}/H$. Its preimage $\widehat{H}$ is a finite-index abelian subgroup of $Z(C_G(H))$ containing $H$, and as $H$ is torsion-free so is $\widehat{H}$. Moreover, as $\widehat{H}$ has finite index in $Z(C_G(H))$, it is $\mathcal{M}''$-quasiconvex, and so by Lemma~\ref{lem:qconv-trans} it is $\mathcal{M}$-quasiconvex in $G$. Thus, $\widehat{H} \leq G$ is an $\mathcal{M}$-quasiconvex free abelian subgroup of finite rank ($N$, say) containing $H$.

Now let $\widehat{\varphi}\colon \widehat{H} \to \R^N$ be an embedding such that $\widehat{\varphi}(\widehat{H}) = \Z^N$. By identifying the subspace of $\R^N$ spanned by $\widehat{\varphi}(H)$ with $\R^n$ via a linear isometry, we see that $\widehat{\varphi}|_H = \theta \circ \varphi$, where $\theta\colon \R^n \to \R^N$ is a linear isometric embedding, and $\varphi\colon H \to \R^n$ is an embedding as a lattice. Given an element $g \in \Comm_G(H)$, we may define a matrix $A_g \in GL_n(\R)$ such that $\varphi(ghg^{-1}) = A_g\varphi(h)$ for all $h \in H \cap g^{-1}Hg$; such a matrix is unique since $H \cap g^{-1}Hg$ has finite index in $H$ and so $\varphi(H \cap g^{-1}Hg)$ is a lattice in $\R^n$. This defines a map
\begin{align*}
\Theta\colon \Comm_G(H) &\to GL_n(\R), \\ g &\mapsto A_g,
\end{align*}
which is easily seen to be a homomorphism---in fact, the map $\Theta$ is just a composite $\Comm_G(H) \to \Comm(H) \cong GL(V) \hookrightarrow GL_n(\R)$, where $V = \varphi(H) \otimes \Q < \R^n$ is an $n$-dimensional $\Q$-vector subspace.

Now by Proposition~\ref{prop:polyhedralaction}, there exists a polyhedral function $\widehat{f}\colon \R^N \to \R$ such that $\widehat{f} \circ \widehat{\varphi}(h) = \widehat{f} \circ \widehat{\varphi}(ghg^{-1})$ for all $g \in G$ and all $h \in \widehat{H} \cap g^{-1}\widehat{H}g$. By Lemma~\ref{lem:polyhedralsubsp}, the function $f = \widehat{f} \circ \theta\colon \R^n \to \R$ is also polyhedral. Now fix $g \in \Comm_G(H)$. Then, for any $h \in H \cap g^{-1}Hg$ we have
\begin{align*}
f(A_g \varphi(h)) &= f \circ \varphi(ghg^{-1}) = \widehat{f} \circ \theta \circ \varphi(ghg^{-1}) = \widehat{f} \circ \widehat{\varphi}(ghg^{-1}) \\ &= \widehat{f} \circ \widehat{\varphi}(h) = \widehat{f} \circ \theta \circ \varphi(h) = f(\varphi(h)),
\end{align*}
and so $f(A_g\mathbf{v}) = f(\mathbf{v})$ for all $\mathbf{v} \in \varphi(H \cap g^{-1}Hg)$. As $f$ is polyhedral, we have $f(\beta \mathbf{v}) = \beta f(\mathbf{v})$ for all $\beta \in [0,\infty)$ and all $\mathbf{v} \in \R^n$, implying that $f(A_g\mathbf{v}) = f(\mathbf{v})$ for all $\mathbf{v} \in K$, where $K := \{ \beta \mathbf{w} \mid \beta \in [0,\infty), \mathbf{w} \in \varphi(H \cap g^{-1}Hg) \}$. As $\varphi(H \cap g^{-1}Hg)$ is a lattice in $\R^n$, the subset $K \subseteq \R^n$ is dense; and as $f$ is polyhedral, it is continuous, implying that $f(A_g\mathbf{v}) = f(\mathbf{v})$ for all $\mathbf{v} \in \R^n$.

Thus, we have $f(A\mathbf{v}) = f(\mathbf{v})$ for all $A \in \Theta(\Comm_G(H))$ and all $\mathbf{v} \in \R^n$. It follows from Proposition~\ref{prop:finite} that $\Theta(\Comm_G(H))$ is finite. Since $\Theta$ factors as a composite $\Comm_G(H) \to \Comm(H) \hookrightarrow GL_n(\R)$, it follows that $\Comm_G(H)$ has finite image in $\Comm(H)$, as required.

The `in particular' part of the Theorem follows directly from the definition of the map $\Comm_G(H) \to \Comm(H)$: indeed, $g \in \Comm_G(H)$ is in the kernel of this map if and only if it centralises a finite-index subgroup of $H$.
\end{proof}

\section*{Data availability}

There is no data associated with the manuscript.

\bibliographystyle{amsalpha}
\bibliography{../../all}

\end{document}